\documentclass{amsart}%
\usepackage{amssymb}
\usepackage{amsfonts}
\usepackage{amsmath}
\usepackage{graphicx}%
\newtheorem{theorem}{Theorem}
\theoremstyle{plain}

\newtheorem{conjecture}{Conjecture}
\newtheorem{corollary}{Corollary}

\newtheorem{proposition}{Proposition}
\newtheorem{remark}{Remark}

\numberwithin{equation}{section}
\begin{document}
\title[Liouville type theorem on manifolds with boundary]{Liouville type theorems on manifolds with nonnegative curvature and strictly
convex boundary }
\author{Qianqiao Guo}
\address{Department of Applied Mathematics, Northwestern Polytechnical University,
Xi'an, Shaanxi 710129, China}
\email{gqianqiao@nwpu.edu.cn}
\author{Fengbo Hang}
\address{Courant Institute, 251 Mercer Street, New York, NY 10012}
\email{fengbo@cims.nyu.edu}
\author{Xiaodong Wang}
\address{Department of Mathematics, Michigan State University, East Lansing, MI 48864}
\email{xwang@math.msu.edu}

\begin{abstract}
We prove some Liouville type theorems on smooth compact Riemannian manifolds with
nonnegative sectional curvature and strictly convex boundary. This gives a
nonlinear generalization in low dimension of the recent sharp lower bound for
the first Steklov eigenvalue by Xia-Xiong and verifies partially a conjecture
by the third named author. As a consequence, we derive several sharp Sobolev trace
inequalities on such manifolds.

\end{abstract}
\maketitle

\section{Introduction\label{sec1}}

In \cite[section 6]{BVV}, a remarkable calculation of Bidaut-V\'{e}ron and
V\'{e}ron implies the following Liouville type theorem (see also \cite{I} for
the case of Neumann boundary condition):

\begin{theorem}
(\cite{BVV, I} ) Let $\left(  M^{n},g\right)  $ be a smooth compact Riemannian
manifold with a (possibly empty) convex boundary. Suppose $u\in C^{\infty
}\left(  M\right)  $ is a positive solution of the following equation%
\[%
\begin{array}
[c]{ccc}%
-\Delta u+\lambda u=u^{q} & \text{on} & M,\\
\frac{\partial u}{\partial\nu}=0 & \text{on} & \partial M,
\end{array}
\]
where $\lambda>0$ is a constant and $1<q\leq\left(  n+2\right)  /\left(
n-2\right)  $. If $Ric\geq\frac{\left(  n-1\right)  \left(  q-1\right)
\lambda}{n}g$, then $u$ must be constant unless $q=\left(  n+2\right)
/\left(  n-2\right)  $ and $\left(  M,g\right)  $ is isometric to $\left(
\mathbb{S}^{n},\frac{4\lambda}{n\left(  n-2\right)  }g_{\mathbb{S}^{n}%
}\right)  $ or $\left(  \mathbb{S}_{+}^{n},\frac{4\lambda}{n\left(
n-2\right)  }g_{\mathbb{S}^{n}}\right)  $. In the latter case $u$ is given on
$\mathbb{S}^{n}$ or $\mathbb{S}_{+}^{n}$\ by the following formula%
\[
u(x)=\frac{1}{\left(  a+x\cdot\xi\right)  ^{\left(  n-2\right)  /2}}.
\]
for some $\xi\in\mathbb{R}^{n+1}$ and some constant $a>\left\vert
\xi\right\vert $.
\end{theorem}

By convex boundary we mean that the 2nd fundamental form $\Pi$ is
nonnegative. To be precise, throughout this paper $\nu$ denotes the outer
unit normal on the boundary and the second fundamental form is defined as
\[
\Pi\left(  X,Y\right)  =\left\langle \nabla_{X}\nu,Y\right\rangle .
\]
for $X,Y\in T_{p}\left(  \partial M\right)  $.

This theorem has some very interesting corollaries. In particular it yields a
sharp lower bound for type I Yamabe invariant (see \cite[section 6]{BVV} and
\cite{W2}). It is proposed in \cite{W2} that a similar result should hold for
type II Yamabe problem on a compact Riemannian manifold with nonnegative Ricci
curvature and strictly convex boundary. By strict convexity we mean the second
fundamental form $\Pi$ of the boundary has a positive lower bound. By scaling
we can always assume that the lower bound is $1$. In its precise form, the
conjecture in \cite{W2} states the following:

\begin{conjecture}
[\cite{W2}]\label{C} Let $\left(  M^{n},g\right)  $ be a smooth compact
Riemannian manifold with $Ric\geq0$ and $\Pi\geq1$ on its nonempty boundary.
Let $u\in C^{\infty}\left(  M\right)  $ be a positive solution to the
following equation%
\begin{equation}%
\begin{array}
[c]{ccc}%
\Delta u=0 & \text{on} & M,\\
\frac{\partial u}{\partial\nu}+\lambda u=u^{q} & \text{on} & \partial M,
\end{array}
\label{Eq}%
\end{equation}
where the parameters $\lambda$ and $q$ are always assumed to satisfy
$\lambda>0$ and $1<q\leq\frac{n}{n-2}$. If $\lambda\leq\frac{1}{q-1}$, then
$u$ must be constant unless $q=\frac{n}{n-2}$, $M$ is isometric to
$\overline{\mathbb{B}^{n}}\subset\mathbb{R}^{n}$ and $u$ corresponds to
\[
u_{a}\left(  x\right)  =\left[  \frac{2}{n-2}\frac{1-\left\vert a\right\vert
^{2}}{1+\left\vert a\right\vert ^{2}\left\vert x\right\vert ^{2}-2x\cdot
a}\right]  ^{\left(  n-2\right)  /2}%
\]
for some $a\in\mathbb{B}^{n}$.
\end{conjecture}

This conjecture, if true, would have very interesting geometric consequences.
We refer the readers to \cite{W2} for further discussion. In this paper we
will verify the conjecture in some special cases.

\begin{theorem}
\label{T1} Let $\left(  M^{n},g\right)  $ be a smooth compact Riemannian
manifold with nonnegative sectional curvature and the second fundamental form
of the boundary $\Pi\geq1$. Then the only positive solution to (\ref{Eq}) is
constant if $\lambda\leq\frac{1}{q-1}$, provided $2\leq n\leq8$ and
$1<q\leq\frac{4n}{5n-9}$.
\end{theorem}

Although this result requires the stronger assumption on the sectional
curvature and severe restriction on the dimension and the exponent, it does
yield the conjectured sharp range for $\lambda$. This is a delicate issue as
illustrated by the followng result on the model space $\overline
{\mathbb{B}^{n}}$.

\begin{proposition}
If $1<q<\frac{n}{n-2}$ and $\lambda\left(  q-1\right)  >1$ then the equation%
\[%
\begin{array}
[c]{ccc}%
\Delta u=0 & \text{on} & \overline{\mathbb{B}^{n}},\\
\frac{\partial u}{\partial\nu}+\lambda u=u^{q} & \text{on} & \partial
\overline{\mathbb{B}^{n}},
\end{array}
\]
admits a positive, nonconstant solution.
\end{proposition}

It should be mentioned that on the model space $\overline{\mathbb{B}^{n}}$
with $n\geq3$ the conjecture is verified in \cite{GuW} in all dimensions when
$\lambda\leq\frac{n-2}{2}$ by the method of moving planes. The approach to
Theorem \ref{T1} is based on an integral method with a key idea borrowed from
the recent work \cite{XX} by Xia and Xiong, where a sharp lower bound for the
first Steklov eigenvalue was proved.

For $n=2$ Theorem \ref{T1} confirms the conjecture when $q\leq8$. By an
approach based on maximum principle in the spirit of \cite{E1, P, W1}, we can
verify the conjecture in dimension 2 for $q\geq2$. Combining both results we fully confirm
Conjecture \ref{C} in dimension $2$.

\begin{theorem}
\label{T2} Let $(\Sigma,g)$ be a smooth compact surface with nonnegative
Gaussian curvature and geodesic curvature $\kappa\geq1$ on the boundary. Then
the only positive solution to the following equation%
\begin{equation}%
\begin{array}
[c]{ccc}%
\Delta u=0 & \text{on} & \Sigma,\\
\frac{\partial u}{\partial\nu}+\lambda u=u^{q} & \text{on} & \partial\Sigma,
\end{array}
\label{E}%
\end{equation}
where $q>1$ and $0<\lambda\leq\frac{1}{q-1}$, is constant.
\end{theorem}

The paper is organized as follows. In Section \ref{sec2} we derive some
integral identities that will be used later. The proof of Theorem \ref{T1} is
given in Section \ref{sec3}. In Section \ref{sec4} we present the argument
based on maximum principle in dimension two and prove Theorem \ref{T2}. In the
last section we make some further remarks about Conjecture \ref{C} and deduce
some corollaries from our Liouville type results.

\textbf{Acknowledgement. }We thank the referee for a careful reading of the
paper. The 3rd named author is partially supported by Simons Foundation
Collaboration Grant for Mathematicians \#312820.

\section{Some integral identities\label{sec2}}

Let $\left(  M^{n},g\right)  $ be a smooth compact Riemannian manifold with
boundary $\Sigma$ and $v\in C^{\infty}\left(  M\right)  $ be a
\textbf{positive} function. We write $f=v|_{\Sigma}$, $\chi=\frac{\partial
v}{\partial\nu}$. Let $w$ be another smooth functions on $M$ satisfying the
following boundary conditions%
\begin{equation}
w|_{\Sigma}=0,\frac{\partial w}{\partial\nu}=-1. \label{ewb}%
\end{equation}

\begin{proposition}
\label{P1}For any $b\in\mathbb{R}$%
\begin{align*}
&  \int_{M}\left(  1-\frac{1}{n}\right)  \left(  \Delta v\right)  ^{2}%
v^{b}w+\frac{b}{2}wv^{b-2}\left\vert \nabla v\right\vert ^{2}\left[  3v\Delta
v+\left(  b-1\right)  \left\vert \nabla v\right\vert ^{2}\right] \\
&  =\int_{M}v^{b}D^{2}w\left(  \nabla v,\nabla v\right)  -\left\vert \nabla
v\right\vert ^{2}v^{b}\Delta w-\frac{b}{2}\left\vert \nabla v\right\vert
^{2}v^{b-1}\left\langle \nabla v,\nabla w\right\rangle \\
&  +\left(  \left\vert D^{2}v-\frac{\Delta v}{n}g\right\vert ^{2}+Ric\left(
\nabla v,\nabla v\right)  \right)  v^{b}w-\int_{\Sigma}f^{b}\left\vert \nabla
f\right\vert ^{2}.
\end{align*}

\end{proposition}

\begin{proof}
The following weighted Reilly formula was proved in \cite{QX} for any smooth
functions $v$ and $\phi$
\begin{align}
&  \int_{M}\left[  \left(  1-\frac{1}{n}\right)  \left(  \Delta v\right)
^{2}-\left\vert D^{2}v-\frac{\Delta v}{n}g\right\vert ^{2}\right]
\phi\label{wr}\\
&  =\int_{M}D^{2}\phi\left(  \nabla v,\nabla v\right)  -\left\vert \nabla
v\right\vert ^{2}\Delta\phi+Ric\left(  \nabla v,\nabla v\right)
\phi\nonumber\\
&  +\int_{\Sigma}\phi\left[  2\chi\Delta_{\Sigma}f+H\chi^{2}+\Pi\left(  \nabla
f,\nabla f\right)  \right]  +\frac{\partial\phi}{\partial\nu}\left\vert \nabla
f\right\vert ^{2},\nonumber
\end{align}
Take $\phi=v^{b}w$. We calculate%
\begin{align*}
\nabla\phi &  =v^{b}\nabla w+bwv^{b-1}\nabla v\\
D^{2}\phi &  =v^{b}D^{2}w+bv^{b-1}\left(  dv\otimes dw+dw\otimes dv\right)
+bwv^{b-1}D^{2}v\\
&  +b\left(  b-1\right)  wv^{b-2}dv\otimes dv,\\
\Delta\phi &  =v^{b}\Delta w+2bv^{b-1}\left\langle \nabla v,\nabla
w\right\rangle +bwv^{b-1}\Delta v+b\left(  b-1\right)  wv^{b-2}\left\vert
\nabla v\right\vert ^{2},\\
D^{2}\phi\left(  \nabla v,\nabla v\right)   &  =v^{b}D^{2}w\left(  \nabla
v,\nabla v\right)  +2bv^{b-1}\left\vert \nabla w\right\vert ^{2}\left\langle
\nabla v,\nabla w\right\rangle +bwv^{b-1}D^{2}v\left(  \nabla v,\nabla
v\right) \\
&  +b\left(  b-1\right)  wv^{b-2}\left\vert \nabla v\right\vert ^{4}.
\end{align*}
Plugging these equations into (\ref{wr}) and using (\ref{ewb}) yields%
\begin{align*}
&  \int_{M}\left[  \left(  1-\frac{1}{n}\right)  \left(  \Delta v\right)
^{2}-\left\vert D^{2}v-\frac{\Delta v}{n}g\right\vert ^{2}\right]  v^{b}w\\
&  =\int_{M}v^{b}D^{2}w\left(  \nabla v,\nabla v\right)  +bwv^{b-1}%
D^{2}v\left(  \nabla v,\nabla v\right)  -\left\vert \nabla v\right\vert
^{2}\left(  v^{b}\Delta w+bwv^{b-1}\Delta v\right) \\
&  +Ric\left(  \nabla v,\nabla v\right)  v^{b}w-\int_{\Sigma}f^{b}\left\vert
\nabla f\right\vert ^{2}.
\end{align*}
We calculate
\begin{align*}
wv^{b-1}D^{2}v\left(  \nabla v,\nabla v\right)   &  =\frac{1}{2}%
wv^{b-1}\left\langle \nabla v,\nabla\left\vert \nabla v\right\vert
^{2}\right\rangle \\
&  =\frac{1}{2}\left[  \mathrm{div}\left(  wv^{b-1}\left\vert \nabla
v\right\vert ^{2}\nabla v\right)  -\left\vert \nabla v\right\vert
^{2}\mathrm{div}\left(  wv^{b-1}\nabla v\right)  \right] \\
&  =\frac{1}{2}\left[  \mathrm{div}\left(  wv^{b-1}\left\vert \nabla
v\right\vert ^{2}\nabla v\right)  -w\left\vert \nabla v\right\vert ^{2}%
v^{b-1}\Delta v\right. \\
&  \left.  -\left(  b-1\right)  wv^{b-2}\left\vert \nabla v\right\vert
^{4}-\left\vert \nabla v\right\vert ^{2}v^{b-1}\left\langle \nabla v,\nabla
w\right\rangle \right]  .
\end{align*}
Integrating yields%
\[
\int_{M}wv^{b-1}D^{2}v\left(  \nabla v,\nabla v\right)  =-\frac{1}{2}\int
_{M}w\left\vert \nabla v\right\vert ^{2}v^{b-1}\Delta v+\left(  b-1\right)
wv^{b-2}\left\vert \nabla v\right\vert ^{4}+\left\vert \nabla v\right\vert
^{2}v^{b-1}\left\langle \nabla v,\nabla w\right\rangle .
\]
Plugging this into the previous integral identity yields%
\begin{align*}
&  \int_{M}\left[  \left(  1-\frac{1}{n}\right)  \left(  \Delta v\right)
^{2}-\left\vert D^{2}v-\frac{\Delta v}{n}g\right\vert ^{2}\right]  v^{b}w\\
&  =\int_{M}v^{b}D^{2}w\left(  \nabla v,\nabla v\right)  -\left\vert \nabla
v\right\vert ^{2}v^{b}\Delta w-\frac{b}{2}wv^{b-2}\left\vert \nabla
v\right\vert ^{2}\left[  3v\Delta v+\left(  b-1\right)  \left\vert \nabla
v\right\vert ^{2}\right] \\
&  -\frac{b}{2}\left\vert \nabla v\right\vert ^{2}v^{b-1}\left\langle \nabla
v,\nabla w\right\rangle +Ric\left(  \nabla v,\nabla v\right)  v^{b}%
w-\int_{\Sigma}f^{b}\left\vert \nabla f\right\vert ^{2}.
\end{align*}
Reorganizing yields the desired identity.
\end{proof}

\begin{proposition}
\label{P2}Under the same assumptions as in Proposition \ref{P1}, we have
\begin{align*}
&  \int_{M}v^{b}D^{2}w\left(  \nabla v,\nabla v\right)  +\left(  v\Delta
v+\frac{b}{2}\left\vert \nabla v\right\vert ^{2}\right)  v^{b-1}\left\langle
\nabla v,\nabla w\right\rangle -\frac{1}{2}v^{b}\left\vert \nabla v\right\vert
^{2}\Delta w\\
&  =\frac{1}{2}\int_{\Sigma}f^{b}\left(  \left\vert \nabla f\right\vert
^{2}-\chi^{2}\right)  .
\end{align*}

\end{proposition}

\begin{proof}
For any vector field $X$ the following identity holds
\[
\left\langle \nabla_{\nabla v}X,\nabla v\right\rangle +Xv\Delta v-\frac{1}%
{2}\left\vert \nabla v\right\vert ^{2}\mathrm{div}X=\mathrm{div}\left(
Xv\nabla v-\frac{1}{2}\left\vert \nabla v\right\vert ^{2}X\right)  .
\]
In the following we take $X=\nabla w$. Note that $\nabla w=-\nu$ on $\Sigma$.
Multiplying both sides of the above identity by $v^{b}$ and integrating
yields
\begin{align*}
&  \int_{M}v^{b}D^{2}w\left(  \nabla v,\nabla v\right)  +v^{b}\Delta
v\left\langle \nabla v,\nabla w\right\rangle -\frac{1}{2}v^{b}\left\vert
\nabla v\right\vert ^{2}\Delta w\\
&  =\int_{M}v^{b}\mathrm{div}\left(  \left\langle \nabla v,\nabla
w\right\rangle \nabla v-\frac{1}{2}\left\vert \nabla v\right\vert ^{2}\nabla
w\right) \\
&  =\int_{M}-bv^{b-1}\left(  \left\langle \nabla v,\nabla w\right\rangle
\left\vert \nabla v\right\vert ^{2}-\frac{1}{2}\left\vert \nabla v\right\vert
^{2}\left\langle \nabla v,\nabla w\right\rangle \right)  +\int_{\Sigma}%
f^{b}\left(  \left\langle \nabla v,\nabla w\right\rangle \chi-\frac{1}%
{2}\left\vert \nabla v\right\vert ^{2}\frac{\partial w}{\partial\nu}\right) \\
&  =-\frac{b}{2}\int_{M}v^{b-1}\left\langle \nabla v,\nabla w\right\rangle
\left\vert \nabla v\right\vert ^{2}+\int_{\Sigma}f^{b}\left(  -\chi^{2}%
+\frac{1}{2}\left\vert \nabla v\right\vert ^{2}\right)  .
\end{align*}
Therefore%
\begin{align*}
&  \int_{M}v^{b}D^{2}w\left(  \nabla v,\nabla v\right)  +\left(  v\Delta
v+\frac{b}{2}\left\vert \nabla v\right\vert ^{2}\right)  v^{b-1}\left\langle
\nabla v,\nabla w\right\rangle -\frac{1}{2}v^{b}\left\vert \nabla v\right\vert
^{2}\Delta w\\
&  =\frac{1}{2}\int_{\Sigma}f^{b}\left(  \left\vert \nabla f\right\vert
^{2}-\chi^{2}\right)  .
\end{align*}

\end{proof}

\section{The proof of Theorem \ref{T1}\label{sec3}}

Throughout this section $\left(  M^{n},g\right)  $ is a smooth compact
Riemannian manifold with nonempty boundary $\Sigma$. We study
\textbf{positive} solutions of the following equation
\[%
\begin{array}
[c]{ccc}%
\Delta u=0 & \text{on} & M,\\
\frac{\partial u}{\partial\nu}+\lambda u=u^{q} & \text{on} & \Sigma,
\end{array}
\]
We write $u=v^{-a}$ with $a\neq0$ a constant to be determined later. Then $v$
satisfies the following equation%
\begin{equation}%
\begin{array}
[c]{ccc}%
\Delta v=\left(  a+1\right)  v^{-1}\left\vert \nabla v\right\vert ^{2} &
\text{on} & M,\\
\chi=\frac{1}{a}\left(  \lambda f-f^{1+a-aq}\right)  & \text{on} & \Sigma,
\end{array}
\label{efv}%
\end{equation}
where $f=v|_{\partial\Sigma}$, $\chi=\frac{\partial v}{\partial\nu}$.
Multiplying both sides by $v^{s}$ and integrating over $M$ yields%
\begin{equation}
\left(  a+s+1\right)  \int_{M}\left\vert \nabla v\right\vert ^{2}v^{s-1}%
=\int_{\Sigma}f^{s}\chi. \label{con}%
\end{equation}
By Proposition \ref{P1}%
\begin{align*}
&  \left[  \left(  1-\frac{1}{n}\right)  \left(  a+1\right)  ^{2}%
+\frac{b\left(  3a+b+2\right)  }{2}\right]  \int_{M}v^{b-2}\left\vert \nabla
v\right\vert ^{4}w\\
&  =\int_{M}v^{b}D^{2}w\left(  \nabla v,\nabla v\right)  -\left\vert \nabla
v\right\vert ^{2}v^{b}\Delta w-\frac{b}{2}\left\vert \nabla v\right\vert
^{2}v^{b-1}\left\langle \nabla v,\nabla w\right\rangle -\int_{\Sigma}%
f^{b}\left\vert \nabla f\right\vert ^{2}+Q,
\end{align*}
where
\[
Q=\int_{M}\left(  \left\vert D^{2}v-\frac{\Delta v}{n}g\right\vert
^{2}+Ric\left(  \nabla v,\nabla v\right)  \right)  v^{b}w.
\]
By Proposition \ref{P2}%
\begin{align*}
&  \int_{M}v^{b}D^{2}w\left(  \nabla v,\nabla v\right)  +\left(  a+1+\frac
{b}{2}\right)  v^{b-2}\left\vert \nabla v\right\vert ^{2}\left\langle \nabla
v,\nabla w\right\rangle -\frac{1}{2}v^{b}\left\vert \nabla v\right\vert
^{2}\Delta w\\
&  =\frac{1}{2}\int_{\Sigma}f^{b}\left(  \left\vert \nabla f\right\vert
^{2}-\chi^{2}\right)  .
\end{align*}
We use the above identity to eliminate the term involving $\left\langle \nabla
v,\nabla w\right\rangle $ in the previous identity and obtain
\begin{align*}
&  \left[  \left(  1-\frac{1}{n}\right)  \left(  a+1\right)  ^{2}%
+\frac{b\left(  3a+b+2\right)  }{2}\right]  \int_{M}v^{b-2}\left\vert \nabla
v\right\vert ^{4}w\\
&  =\int_{M}\frac{a+1+b}{a+1+\frac{b}{2}}v^{b}D^{2}w\left(  \nabla v,\nabla
v\right)  -\frac{a+1+\frac{3}{4}b}{a+1+\frac{b}{2}}\left\vert \nabla
v\right\vert ^{2}v^{b}\Delta w\\
&  +\int_{\Sigma}\frac{\frac{1}{4}b}{a+1+\frac{b}{2}}f^{b}\chi^{2}%
-\frac{a+1+\frac{3}{4}b}{a+1+\frac{b}{2}}f^{b}\left\vert \nabla f\right\vert
^{2}+Q.
\end{align*}
We choose $b=-\frac{4}{3}\left(  a+1\right)  $. Then%
\begin{align}
&  \frac{\left[  5n-9-\left(  n+9\right)  a\right]  \left(  a+1\right)  }%
{9n}\int_{M}v^{b-2}\left\vert \nabla v\right\vert ^{4}w\label{kii}\\
&  =-\int_{M}v^{b}D^{2}w\left(  \nabla v,\nabla v\right)  -\int_{\Sigma}%
f^{b}\chi^{2}+Q.\nonumber
\end{align}

Let $\rho=d\left(  \cdot,\Sigma\right)  $ be the distance function to the
boundary. It is Lipschitz on $M$ and smooth \ away from the cut locus
$\mathrm{Cut}\left(  \Sigma\right)  $ which is a closed set of measure zero in
the interior of $M$. We consider $\psi:=\rho-\frac{\rho^{2}}{2}$. Notice that
$\psi$ is smooth near $\Sigma$ and satisfies%
\[
\psi|_{\Sigma}=0,\frac{\partial\psi}{\partial\nu}=-1.
\]
From now on we assume that $M$ has nonnegative sectional curvature and
$\Pi\geq1$ on $\Sigma$. By the Hessian comparison theorem (cf. \cite{K})
$\rho\leq1$ hence $\psi\geq0$ and
\[
-D^{2}\psi\geq g
\]
in the support sense. The new idea that $\psi$ can be used as a good weight
function is introduced in \cite{XX} to study the first Steklov eigenvalue. To
overcome the difficulty that $\psi$ is not smooth, they constructed smooth approximations.

\begin{proposition}
[\cite{XX}]Fix a neighborhood $\mathcal{C}$ of $\mathrm{Cut}\left(
\Sigma\right)  $ in the interior of $M$. Then for any $\varepsilon>0$, there
exists a smooth nonnegative function $\psi_{\varepsilon}$ on $M$ s.t.
$\psi_{\varepsilon}=\psi$ on $M\backslash\mathcal{C}$ and
\[
-D^{2}\psi_{\varepsilon}\geq\left(  1-\varepsilon\right)  g
\]

\end{proposition}

The construction is based on the work \cite{GW1, GW2, GW3}.

In (\ref{kii}) taking the weight $w=\psi_{\varepsilon}$ yields%
\begin{align*}
&  \frac{\left[  5n-9-\left(  n+9\right)  a\right]  \left(  a+1\right)  }%
{9n}\int_{M}v^{b-2}\left\vert \nabla v\right\vert ^{4}\psi_{\varepsilon}\\
&  \geq\left(  1-\varepsilon\right)  \int_{\mathcal{C}}v^{b}\left\vert \nabla
v\right\vert ^{2}-\int_{M\backslash\mathcal{C}}v^{b}D^{2}\psi\left(  \nabla
v,\nabla v\right)  -\int_{\Sigma}f^{b}\chi^{2}+Q_{\varepsilon},
\end{align*}
where
\[
Q_{\varepsilon}=\int_{M}\left(  \left\vert D^{2}v-\frac{\Delta v}%
{n}g\right\vert ^{2}+Ric\left(  \nabla v,\nabla v\right)  \right)  v^{b}%
\psi_{\varepsilon}.
\]
Letting $\varepsilon\rightarrow0$ and shrinking the neighborhood yields
\begin{align*}
&  \frac{\left[  5n-9-\left(  n+9\right)  a\right]  \left(  a+1\right)  }%
{9n}\int_{M}v^{b-2}\left\vert \nabla v\right\vert ^{4}\psi\\
&  \geq\int_{\mathcal{C}}v^{b}\left\vert \nabla v\right\vert ^{2}%
-\int_{M\backslash\mathcal{C}}v^{b}D^{2}\psi\left(  \nabla v,\nabla v\right)
-\int_{\Sigma}f^{b}\chi^{2}+Q
\end{align*}
where
\[
Q=\int_{M}\left(  \left\vert D^{2}v-\frac{\Delta v}{n}g\right\vert
^{2}+Ric\left(  \nabla v,\nabla v\right)  \right)  v^{b}\psi.
\]
On $M\backslash\mathcal{C}$ the function $\psi$ is smooth and satisfies
$-D^{2}\psi\geq g$. Therefore%
\begin{align*}
&  \frac{\left[  5n-9-\left(  n+9\right)  a\right]  \left(  a+1\right)  }%
{9n}\int_{M}v^{b-2}\left\vert \nabla v\right\vert ^{4}\psi\\
&  \geq\int_{M}v^{b}\left\vert \nabla v\right\vert ^{2}-\int_{\Sigma}f^{b}%
\chi^{2}+Q.
\end{align*}
Using the boundary condition for $v$ we obtain%
\begin{align*}
&  \frac{\left[  5n-9-\left(  n+9\right)  a\right]  \left(  a+1\right)  }%
{9n}\int_{M}v^{b-2}\left\vert \nabla v\right\vert ^{4}\psi\\
&  \geq\int_{M}v^{b}\left\vert \nabla v\right\vert ^{2}-\frac{1}{a}%
\int_{\Sigma}\left(  \lambda f^{b+1}-f^{b+1+a-aq}\right)  \chi+Q\\
&  =\int_{M}v^{b}\left\vert \nabla v\right\vert ^{2}-\frac{\left(
a+b+2\right)  \lambda}{a}w^{b}\left\vert \nabla w\right\vert ^{2}%
+\frac{2a+b+2-aq}{a}w^{b+a-aq}\left\vert \nabla w\right\vert ^{2}+Q\\
&  =\int_{M}\left(  1-\frac{\lambda\left(  2-a\right)  }{3a}\right)
v^{b}\left\vert \nabla v\right\vert ^{2}+\left(  \frac{2}{3}-q+\frac{2}%
{3a}\right)  v^{b+a-aq}\left\vert \nabla v\right\vert ^{2}+Q,
\end{align*}
which can be written as
\begin{equation}
A\int_{M}v^{b-2}\left\vert \nabla v\right\vert ^{4}\psi+B\int_{M}%
v^{b}\left\vert \nabla v\right\vert ^{2}+C\int_{M}v^{b+a-aq}\left\vert \nabla
v\right\vert ^{2}\geq Q, \label{Fi}%
\end{equation}
where, with $x=a^{-1}$%
\begin{align*}
A  &  =\frac{\left[  5n-9-\left(  n+9\right)  a\right]  \left(  a+1\right)
}{9n}=\frac{\left[  \left(  5n-9\right)  x-\left(  n+9\right)  \right]
\left(  x+1\right)  }{9nx^{2}},\\
B  &  =\frac{\lambda\left(  2-a\right)  }{3a}-1=\frac{\lambda}{3}\left(
2x-1\right)  -1\\
C  &  =q-\frac{2}{3}-\frac{2}{3a}=q-\frac{2}{3}-\frac{2}{3}x
\end{align*}
We want to choose $a$ s.t. $A,B,C\leq0$, i.e.
\begin{align*}
\left(  x-\frac{n+9}{5n-9}\right)  \left(  x+1\right)   &  \leq0,\\
\frac{\lambda}{3}\left(  2x-1\right)  -1  &  \leq0,\\
q-\frac{2}{3}-\frac{2}{3}x  &  \leq0.
\end{align*}
By simple calculations these inequalities become%
\begin{align*}
-1  &  \leq x\leq\frac{n+9}{5n-9},\\
\frac{3}{2}q-1  &  \leq x\leq\frac{3}{2}\frac{1}{\lambda}+\frac{1}{2}.
\end{align*}
The choice is possible when $\frac{3}{2}q-1\leq\frac{3}{2}\frac{1}{\lambda
}+\frac{1}{2}$ and $\frac{3}{2}q-1\leq\frac{n+9}{5n-9}$ i.e. when $\left(
q-1\right)  \lambda\leq1$ and $q\leq\frac{4n}{5n-9}$. As $q>1$ we must have
$2\leq n\leq8$. Then when $q\leq\frac{4n}{5n-9}$ and $\left(  q-1\right)
\lambda\leq1$ by choosing $\frac{1}{a}=\frac{3}{2}q-1$ we have
\[
C=0,B=\left(  q-1\right)  \lambda-1\leq0,A=\frac{5n-9}{6n}q\left(  \frac{3}%
{2}q-1\right)  ^{2}\left(  q-\frac{4n}{5n-9}\right)  \leq0.
\]
Thus the left hand side of (\ref{Fi}) is nonpositive while the right hand side
is nonnegative. It follows that both sides of (\ref{Fi}) are zero and we must
have%
\begin{equation}
D^{2}v=\frac{a+1}{n}v^{-1}\left\vert \nabla v\right\vert ^{2}g,\quad
Ric\left(  \nabla v,\nabla v\right)  =0. \label{od}%
\end{equation}
If $q<\frac{4n}{5n-9}$ or $\lambda\left(  q-1\right)  <1$ we have $A<0$ or
$B<0$, respectively and hence $v$ must be constant. It remains to prove that
$v$ must also be constant when
\begin{equation}
q=\frac{4n}{5n-9},\quad\lambda\left(  q-1\right)  =1. \label{nm}%
\end{equation}
Under this assumption, we have%
\[
a=\frac{1}{\frac{3}{2}q-1}=\frac{5n-9}{n+9}.
\]
As $Ric\geq0$ the second equation in (\ref{od}) implies $Ric\left(  \nabla
v,\cdot\right)  =0$. We denote
\[
h=\frac{a+1}{n}v^{-1}\left\vert \nabla v\right\vert ^{2}=\frac{6}{n+9}%
v^{-1}\left\vert \nabla v\right\vert ^{2}.
\]
Then $D^{2}v=hg$. Working with a local orthonormal frame we differentiate%
\begin{align*}
h_{j}  &  =v_{ij,i}=v_{ii,j}-R_{jiil}v_{l}\\
&  =\left(  \Delta v\right)  _{j}+R_{jl}v_{l}\\
&  =nh_{j}.
\end{align*}
Thus $h_{j}=0$, i.e. $h$ is constant. To continue, we observe that since
\[
\left\vert \nabla v\right\vert ^{2}=\frac{n+9}{6}hv,
\]
differentiating both sides we get%
\[
\frac{n+9}{6}hv_{j}=2v_{i}v_{ij}=2hv_{j}.
\]
Therefore%
\[
\left(  n-3\right)  h\nabla v=0.
\]
Taking inner product on both sides with $\nabla v$ and using the fact $v>0$,
we see $\left(  n-3\right)  h^{2}=0$. When $n\neq3$, we have $h=0$ and hence
$\nabla v=0$ and $v$ must be a constant function.

It remains to handle the case $n=3$, $q=2$ and $\lambda=1$. We need to further
inspect the proof and observe that we used the inequality $-D^{2}\psi\left(
\nabla v,\nabla v\right)  \geq\left\vert \nabla v\right\vert ^{2}$ on
$M\backslash\mathcal{C}$. Therefore this must be an equality. Then this
implies that
\[
-D^{2}\psi\left(  \nabla v,\cdot\right)  =\left\langle \nabla v,\cdot
\right\rangle .
\]
As $-\nabla\psi=\nu$ on the boundary the above identity implies $\Pi\left(
\nabla f,\cdot\right)  =\left\langle \nabla f,\cdot\right\rangle $ on $\Sigma
$. As $D^{2}v=hg$ we have for $X\in T\Sigma$%
\begin{align*}
0  &  =D^{2}v\left(  X,\nu\right) \\
&  =X\chi-\Pi\left(  \nabla f,X\right) \\
&  =X\chi-Xf.
\end{align*}
Thus $\chi-f$ is constant. But as $\chi=2\left(  f-f^{1/2}\right)  $ by the
boundary condition we conclude $f$ is constant. Therefore $v$ is constant.

\section{Maximum principle argument in dimension 2\label{sec4}}

It is unfortunate that the integral argument in previous section only works
for $1<q\leq8$ in dimension $2$. On the other hand, in \cite{E1, P}, an
approach based on maximum principle is developed to derive a sharp lower bound
of the first Steklov eigenvalue on a compact surface with boundary. This idea
is also used in \cite{W1} to prove the limiting case $q=\infty$. Surprisingly
this type of argument works for any power $q\geq2$.

Throughout this section $(\Sigma,g)$ is a smooth compact surface with Gaussian
curvature $K\geq0$ and geodesic curvature $\kappa\geq1$ on the boundary. Our
goal is to prove the following uniqueness result.

\begin{theorem}
\label{m'} Let $u>0$ be a smooth function on $\Sigma$ satisfying the following
equation%
\[%
\begin{array}
[c]{ccc}%
\Delta u=0 & \text{on} & \Sigma,\\
\frac{\partial u}{\partial\nu}+\lambda u=u^{q} & \text{on} & \partial\Sigma,
\end{array}
\]
where $\lambda$ is a positive constant and $q\geq2$. Then $u$ must be a
constant function if $\lambda\leq\frac{1}{q-1}$.
\end{theorem}

Theorem \ref{T2} follows by combining the above theorem and Theorem \ref{T1}.

To prove Theorem \ref{m'} we write $u=v^{-a}$, with $a\neq0$ to be determined.
Then $v$ satisfies
\[%
\begin{array}
[c]{ccc}%
\Delta v=\left(  a+1\right)  v^{-1}\left\vert \nabla v\right\vert ^{2} &
\text{on} & \Sigma,\\
\chi=\frac{1}{a}\left(  \lambda f-f^{1+a-aq}\right)  & \text{on} &
\partial\Sigma,
\end{array}
\]
where $f=v|_{\partial\Sigma},\chi=\frac{\partial v}{\partial\nu}$. Let
$\phi=v^{b}\left\vert \nabla v\right\vert ^{2}$ with $b$ to be determined.

\begin{proposition}
We have%
\begin{equation}
\Delta\phi-2\left(  a+b+1\right)  v^{-1}\left\langle \nabla v,\nabla
\phi\right\rangle \geq\left[  a\left(  a-b\right)  -\left(  b+1\right)
^{2}\right]  v^{-b-2}\phi^{2}. \label{inephi}%
\end{equation}

\end{proposition}

\begin{proof}
We have $\left\vert \nabla v\right\vert ^{2}=v^{-b}\phi$. We compute%
\begin{align*}
\Delta\left\vert \nabla v\right\vert ^{2}  &  =v^{-b}\Delta\phi-2bv^{-b-1}%
\left\langle \nabla v,\nabla\phi\right\rangle +\phi\Delta v^{-b}\\
&  =v^{-b}\Delta\phi-2bv^{-b-1}\left\langle \nabla v,\nabla\phi\right\rangle
+\phi\left[  -bv^{-b-1}\Delta v+b\left(  b+1\right)  v^{-b-2}\left\vert \nabla
v\right\vert ^{2}\right] \\
&  =v^{-b}\Delta\phi-2bv^{-b-1}\left\langle \nabla v,\nabla\phi\right\rangle
+b\left(  b-a\right)  v^{-2b-2}\phi^{2}.
\end{align*}
Using the Bochner formula we obtain%
\begin{align*}
&  v^{-b}\Delta\phi-2bv^{-b-1}\left\langle \nabla v,\nabla\phi\right\rangle
+b\left(  b-a\right)  v^{-2b-2}\phi^{2}\\
&  \geq2\left\vert D^{2}v\right\vert ^{2}+2\left\langle \nabla v,\nabla\Delta
v\right\rangle \\
&  \geq\left(  \Delta v\right)  ^{2}+2\left\langle \nabla v,\nabla\Delta
v\right\rangle \\
&  =\left(  a+1\right)  ^{2}v^{-2b-2}\phi^{2}+2\left(  a+1\right)  \left[
v^{-b-1}\left\langle \nabla v,\nabla\phi\right\rangle -\left(  b+1\right)
v^{-2b-2}\phi^{2}\right] \\
&  =\left(  a+1\right)  \left(  a-2b-1\right)  v^{-2b-2}\phi^{2}+2\left(
a+1\right)  v^{-b-1}\left\langle \nabla v,\nabla\phi\right\rangle .
\end{align*}
Therefore
\[
\Delta\phi-2\left(  a+b+1\right)  v^{-1}\left\langle \nabla v,\nabla
\phi\right\rangle \geq\left[  a\left(  a-b\right)  -\left(  b+1\right)
^{2}\right]  v^{-b-2}\phi^{2}.
\]

\end{proof}

We impose the following condition on $a$ and $b$%
\begin{equation}
a\left(  a-b\right)  -\left(  b+1\right)  ^{2}>0. \label{subphi}%
\end{equation}

As a result, $\Delta\phi-2\left(  a+b+1\right)  v^{-1}\left\langle \nabla
v,\nabla\phi\right\rangle \geq0$. By the maximum principle, $\phi$ achieves
its maximum somewhere on the boundary. We use the arclength $s$ to parametrize
the boundary. Suppose that $\phi$ achieves its maximum at $s_{0} $ on the
boundary. Then we have%
\[
\phi^{\prime}\left(  s_{0}\right)  =0,\phi^{\prime\prime}\left(  s_{0}\right)
\leq0,\frac{\partial\phi}{\partial\nu}\left(  s_{0}\right)  \geq0.
\]
Moreover by the Hopf lemma, the 3rd inequality is strict unless $\phi$ is constant.

\begin{proposition}
We have%
\[
\frac{\partial\phi}{\partial\nu}\leq2f^{b}\left[  \left(  \left(  \frac{b}%
{2}+a+1\right)  \frac{\chi}{f}-1\right)  \left(  \left(  f^{\prime}\right)
^{2}+\chi^{2}\right)  +f^{\prime}\chi^{\prime}-\chi f^{\prime\prime}\right]
.
\]

\end{proposition}

\begin{proof}
We compute%
\begin{align*}
\frac{\partial\phi}{\partial\nu}  &  =2f^{b}D^{2}v\left(  \nabla v,\nu\right)
+bf^{b-1}\chi\left(  \left(  f^{\prime}\right)  ^{2}+\chi^{2}\right) \\
&  =2f^{b}\left[  \chi D^{2}v\left(  \nu,\nu\right)  +f^{\prime}D^{2}v\left(
\frac{\partial}{\partial s},\nu\right)  +\frac{b\chi}{2f}\left(  \left(
f^{\prime}\right)  ^{2}+\chi^{2}\right)  \right]  .
\end{align*}
On one hand
\begin{align*}
D^{2}v\left(  \frac{\partial}{\partial s},\nu\right)   &  =\left\langle
\nabla_{\frac{\partial}{\partial s}}\nabla v,\nu\right\rangle \\
&  =\chi^{\prime}-\left\langle \nabla v,\nabla_{\frac{\partial}{\partial s}%
}\nu\right\rangle \\
&  =\chi^{\prime}-f^{\prime}\left\langle \frac{\partial}{\partial s}%
,\nabla_{\frac{\partial}{\partial s}}\nu\right\rangle \\
&  =\chi^{\prime}-\kappa f^{\prime}.
\end{align*}
On the other hand from the equation of $v$ we have on $\partial\Sigma$%
\[
D^{2}v\left(  \nu,\nu\right)  +\kappa\chi+f^{\prime\prime}=\left(  a+1\right)
f^{-1}\left(  \left(  f^{\prime}\right)  ^{2}+\chi^{2}\right)  .
\]
Plugging the above two identities into the formula for $\frac{\partial\phi
}{\partial\nu}$ yields%
\[
\frac{\partial\phi}{\partial\nu}=2f^{b}\left[  \left(  \left(  \frac{b}%
{2}+a+1\right)  \frac{\chi}{f}-\kappa\right)  \left(  \left(  f^{\prime
}\right)  ^{2}+\chi^{2}\right)  +f^{\prime}\chi^{\prime}-\chi f^{\prime\prime
}\right]  .
\]
where in the last step we use the assumption $\kappa\geq1$.
\end{proof}

As
\[
\phi\left(  s\right)  :=\phi|_{\partial\Sigma}=f\left(  s\right)  ^{b}\left(
f^{\prime}\left(  s\right)  ^{2}+\chi\left(  s\right)  ^{2}\right)  ,
\]
we obtain%
\[
\phi^{\prime}\left(  s\right)  =2f^{b}f^{\prime}\left[  f^{\prime\prime}%
+\frac{1}{a}\chi\left(  \lambda-\left(  1+a-aq\right)  f^{a-aq}\right)
+\frac{b}{2f}\left(  f^{\prime}{}^{2}+\chi^{2}\right)  \right]  .
\]
If $f^{\prime}\left(  s_{0}\right)  \neq0$ then at $s_{0}$%
\[
f^{\prime\prime}=-\frac{1}{a}\chi\left(  \lambda-\left(  1+a-aq\right)
f^{a-aq}\right)  -\frac{b}{2f}\left(  f^{\prime}{}^{2}+\chi^{2}\right)  .
\]
Therefore
\begin{align*}
\frac{\partial\phi}{\partial\nu}  &  \leq2f^{b}\left[  \left(  \left(
\frac{b}{2}+a+1\right)  \frac{\chi}{f}-1\right)  \left(  \left(  f^{\prime
}\right)  ^{2}+\chi^{2}\right)  +f^{\prime}\chi^{\prime}\right. \\
&  \left.  +\frac{1}{a}\chi^{2}\left(  \lambda-\left(  1+a-aq\right)
f^{a-aq}\right)  +\frac{b}{2}\frac{\chi}{f}\left(  f^{\prime}{}^{2}+\chi
^{2}\right)  \right] \\
&  =2f^{b}\left(  \left(  f^{\prime}\right)  ^{2}+\chi^{2}\right)  \left[
\frac{a+b+1}{a}\left(  \lambda-f^{a-aq}\right)  -1+\frac{1}{a}\left(
\lambda-\left(  1+a-aq\right)  f^{a-aq}\right)  \right] \\
&  =2f^{b}\left(  \left(  f^{\prime}\right)  ^{2}+\chi^{2}\right)  \left[
\frac{a+b+2}{a}\lambda-1-\frac{\left(  2-q\right)  a+b+2}{a}f^{a-aq}\right]  .
\end{align*}
We want%
\begin{align*}
\frac{a+b+2}{a}\lambda-1  &  \leq0,\\
\left(  2-q\right)  a+b+2  &  =0.
\end{align*}
Therefore we choose $b=\left(  q-2\right)  a-2$. Then the 1st equation is
simply $\left(  q-1\right)  \lambda\leq1$. The condition (\ref{subphi})
becomes%
\[
\left(  q^{2}-3q+1\right)  a^{2}-2\left(  q-1\right)  a+1<0.
\]
A solution always exists as the discriminant equals $4q>0$. Under such choices
for $a$ and $b$ we have $\frac{\partial\phi}{\partial\nu}\left(  s_{0}\right)
\leq0$. Therefore $\phi$ is constant.

If $f^{\prime}\left(  s_{0}\right)  =0$ then at $s_{0}$%
\[
\phi^{\prime\prime}\left(  s_{0}\right)  =2f^{b}f^{\prime\prime}\left[
f^{\prime\prime}+\frac{1}{a}\chi\left(  \lambda-\left(  1+a-aq\right)
f^{a-aq}\right)  +\frac{b}{2}\frac{\chi^{2}}{f}\right]  \leq0.
\]
Therefore we have at $s_{0}$%
\[
\left(  f^{\prime\prime}\right)  ^{2}+f^{\prime\prime}\chi\left[  \left(
q-1\right)  \lambda-\frac{qa}{2}\frac{\chi}{f}\right]  \leq0.
\]
while the condition $\frac{\partial\phi}{\partial\nu}\left(  s_{0}\right)
\geq0$ becomes
\[
\left(  \frac{qa}{2}\frac{\chi}{f}-1\right)  \chi^{2}-\chi f^{\prime\prime
}\geq0.
\]
Set $A=\left(  q-1\right)  \lambda-\frac{qa}{2}\frac{\chi}{f}$. We have
$\frac{qa}{2}\frac{\chi}{f}-1\leq\frac{qa}{2}\frac{\chi}{f}-\left(
q-1\right)  \lambda=-A$. Therefore the above two inequalities imply%
\begin{align*}
\chi\left(  A\chi+f^{\prime\prime}\right)   &  \leq0,\\
f^{\prime\prime}\left(  A\chi+f^{\prime\prime}\right)   &  \leq0.
\end{align*}
We have%
\[
A=\left(  q-1\right)  \lambda-\frac{q}{2}\left(  \lambda-f^{a-qa}\right)
=\left(  \frac{q}{2}-1\right)  \lambda+\frac{q}{2}f^{a-qa}\geq0
\]
if $q\geq2$. Combining the two inequalities we then get $\left(
A\chi+f^{\prime\prime}\right)  ^{2}\leq0$. Therefore $A\chi+f^{\prime\prime
}=0$. Then again we have $\frac{\partial\phi}{\partial\nu}\left(
s_{0}\right)  \leq0$ and $\phi$ must be constant.

In all cases we have proved that $\phi$ is constant. As the coefficient on the
right hand side of (\ref{inephi}) is positive, we must have $\phi\equiv0$.
Therefore $u$ is constant. This finishes the proof of Theorem \ref{m'}.

\section{Further discussions\label{sec5}}

Let $\left(  M^{n},g\right)  $ be a smooth compact Riemannian manifold with
boundary $\Sigma$. We consider for $1<q\leq\frac{n}{n-2}$ and $\lambda>0$ the
functional%
\[
J_{q,\lambda}\left(  u\right)  =\frac{\int_{M}\left\vert \nabla u\right\vert
^{2}+\lambda\int_{\Sigma}u^{2}}{\left(  \int_{\Sigma}\left\vert u\right\vert
^{q+1}\right)  ^{\frac{2}{q+1}}},\text{ }u\in H^{1}\left(  M\right)
\backslash\left\{  0\right\}  .
\]
The first variation in the direction of $\overset{\cdot}{u}$ is
\begin{align*}
&  2\left[  \frac{\int_{M}\left\langle \nabla u,\nabla\overset{\cdot}%
{u}\right\rangle +\lambda\int_{\Sigma}u\overset{\cdot}{u}}{\left(
\int_{\Sigma}\left\vert u\right\vert ^{q+1}\right)  ^{\frac{2}{q+1}}}%
-\frac{\int_{M}\left\vert \nabla u\right\vert ^{2}+\lambda\int_{\Sigma}u^{2}%
}{\left(  \int_{\Sigma}\left\vert u\right\vert ^{q+1}\right)  ^{1+\frac
{2}{q+1}}}\int_{\Sigma}\left\vert u\right\vert ^{q}\overset{\cdot}{u}\right]
\\
&  =\frac{2}{\left(  \int_{\Sigma}\left\vert u\right\vert ^{q+1}\right)
^{\frac{2}{q+1}}}\left[  -\int_{M}\overset{\cdot}{u}\Delta u+\int_{\Sigma
}\left(  \frac{\partial u}{\partial\nu}+\lambda u\right)  \overset{\cdot}%
{u}-\frac{\int_{M}\left\vert \nabla u\right\vert ^{2}+\lambda\int_{\Sigma
}u^{2}}{\int_{\Sigma}\left\vert u\right\vert ^{q+1}}\int_{\Sigma}\left\vert
u\right\vert ^{q}\overset{\cdot}{u}\right]
\end{align*}
Thus a positive $u$ is a critical point iff%
\[%
\begin{array}
[c]{ccc}%
\Delta u=0 & \text{on} & M,\\
\frac{\partial u}{\partial\nu}+\lambda u=cu^{q} & \text{on} & \Sigma,
\end{array}
\]
with $c=\frac{\int_{M}\left\vert \nabla u\right\vert ^{2}+\lambda\int_{\Sigma
}u^{2}}{\int_{\Sigma}\left\vert u\right\vert ^{q+1}}$. In particular
$u_{0}\equiv1$ is a critical point. The second variation at $u_{0}$ in the
direction of $\overset{\cdot}{u}$ with $\int_{\Sigma}\overset{\cdot}{u}=0$ is
\begin{align*}
&  \frac{2}{\left\vert \Sigma\right\vert ^{\frac{2}{q+1}}}\left[  -\int
_{M}\overset{\cdot}{u}\Delta\overset{\cdot}{u}+\int_{\Sigma}\left(
\frac{\partial\overset{\cdot}{u}}{\partial\nu}+\lambda\overset{\cdot}%
{u}\right)  \overset{\cdot}{u}-\lambda q\left(  \overset{\cdot}{u}\right)
^{2}\right] \\
&  =\frac{2}{\left\vert \Sigma\right\vert ^{\frac{2}{q+1}}}\left[  \int
_{M}\left\vert \nabla\overset{\cdot}{u}\right\vert ^{2}-\lambda\left(
q-1\right)  \int_{\Sigma}\left(  \overset{\cdot}{u}\right)  ^{2}\right]  .
\end{align*}
Therefore $u_{0}$ is stable iff $\lambda\left(  q-1\right)  \leq\sigma_{1}$,
the first Steklov eigenvalue. On $\overline{\mathbb{B}^{n}}$ the first Steklov
eigenvalue is $1$. Therefore $u_{0}$ is not stable on $\overline
{\mathbb{B}^{n}}$ \ when $\lambda\left(  q-1\right)  >1$. As the trace
operator $H^{1}\left(  M\right)  \rightarrow L^{q}\left(  \Sigma\right)  $ is
compact when $q<\frac{n}{n-2}$, $\inf J_{q,\lambda}$ is always achieved.
Therefore we get the following

\begin{proposition}
If $q<\frac{n}{n-2}$ and $\lambda\left(  q-1\right)  >1$ then the equation%
\[%
\begin{array}
[c]{ccc}%
\Delta u=0 & \text{on} & \overline{\mathbb{B}^{n}},\\
\frac{\partial u}{\partial\nu}+\lambda u=u^{q} & \text{on} & \partial
\overline{\mathbb{B}^{n}},
\end{array}
\]
admits a positive, nonconstant solution.
\end{proposition}

In the general case, under the assumption that $Ric\geq0$ and $\Pi\geq1$ on
$\Sigma$, Conjecture \ref{C} claims that $u_{0}$, up to scaling, is the only
positive critical point of $J_{q,\lambda}$ if $\lambda\left(  q-1\right)
\leq1$. In particular we must have $\sigma_{1}\geq1$ if the conjecture is true
for a single exponent $q$. Therefore Conjecture \ref{C} implies the following
conjecture of Escobar \cite{E2} .

\begin{conjecture}
[\cite{E2}]Let $\left(  M^{n},g\right)  $ be a compact Riemannian manifold
with boundary with $Ric\geq0$ and $\Pi\geq1$ on $\Sigma$. Then the 1st Steklov
eigenvalue $\sigma_{1}\geq1$.
\end{conjecture}

In \cite{E1}, the conjecture is confirmed when $n=2$, extending the method of
\cite{P}, where the same estimate for a planar domain is derived. In other
dimensions, under the stronger assumption that $M$ has nonnegative sectional
curvature, the conjecture was proved recently in \cite{XX}. By the previous
discussion, Theorem \ref{T1} implies estimate in \cite{XX} when $2\leq n\leq8$
and can be viewed as a nonlinear generalization. Theorem \ref{T1} also gives
us the following sharp Sobolev inequalities (see also the discussions in
\cite{W2}).

\begin{corollary}
Let $\left(  M^{n},g\right)  $ be a smooth compact Riemannian manifold with
nonnegative sectional curvature and $\Pi\geq1$ on the boundary $\Sigma$.
Assume $2\leq n\leq8$ and $1<q\leq\frac{4n}{5n-9}$. Then
\begin{equation}
\left(  \frac{1}{\left\vert \Sigma\right\vert }\int_{\Sigma}\left\vert
u\right\vert ^{q+1}\right)  ^{2/\left(  q+1\right)  }\leq\frac{q-1}{\left\vert
\Sigma\right\vert }\int_{M}\left\vert \nabla u\right\vert ^{2}+\frac
{1}{\left\vert \Sigma\right\vert }\int_{\Sigma}u^{2}. \label{ineq_t}%
\end{equation}

\end{corollary}

In the limiting case we can deduce the following logarithmic inequality.

\begin{corollary}
Let $\left(  M^{n},g\right)  $ be a compact Riemannian manifold with
nonnegative sectional curvature and $\Pi\geq1$ on the boundary $\Sigma$.
Assume $2\leq n\leq8$. Then for any $u\in C^{\infty}\left(  M\right)  $ with
$\frac{1}{\left\vert \Sigma\right\vert }\int_{\Sigma}u^{2}=1$, we have%
\[
\frac{1}{\left\vert \Sigma\right\vert }\int_{\Sigma}\left\vert u\right\vert
^{2}\log u^{2}\leq\frac{2}{\left\vert \Sigma\right\vert }\int_{M}\left\vert
\nabla u\right\vert ^{2}.
\]

\end{corollary}

\begin{proof}
Under the assumption on $u$ (\ref{ineq_t}) can be written as%
\[
\frac{1}{q-1}\left[  \left(  \frac{1}{\left\vert \Sigma\right\vert }%
\int_{\Sigma}\left\vert u\right\vert ^{q+1}\right)  ^{2/\left(  q+1\right)
}-1\right]  \leq\frac{1}{\left\vert \Sigma\right\vert }\int_{M}\left\vert
\nabla u\right\vert ^{2}.
\]
Taking limit $q\downarrow1$ and applying L'Hospital's rule yields the desired inequality.
\end{proof}

\begin{remark}
Linearization of the above inequality around $u_{0}\equiv1$ yields the
inequality $\sigma_{1}\geq1$, i.e. if $\int_{\Sigma}u=0$, then
\[
\int_{\Sigma}u^{2}\leq\int_{M}\left\vert \nabla u\right\vert ^{2}.
\]

\end{remark}

In dimension two we have a complete result in Theorem \ref{T2}. As a corollary
we have

\begin{corollary}
\label{inqf}Let $(\Sigma,g)$ be a smooth compact surface with nonnegative
Gaussian curvature and geodesic curvature $\kappa\geq1$. Then for any $u\in
H^{1}\left(  \Sigma\right)  $ and $q\geq1$, we have
\[
L^{\left(  q-1\right)  /\left(  q+1\right)  }\left(  \int_{\partial\Sigma
}\left\vert u\right\vert ^{q+1}\right)  ^{2/\left(  q+1\right)  }\leq\left(
q-1\right)  \int_{\Sigma}\left\vert \nabla u\right\vert ^{2}+\int
_{\partial\Sigma}u^{2}.
\]
Here $L$ is the length of $\partial\Sigma$. Moreover, equality holds iff $u $
is a constant function.
\end{corollary}

Finally we recall the following Moser-Trudinger-Onofri type inequality on the
disc $\overline{\mathbb{B}^{2}}$ derived in \cite{OPS}: for any $u\in
H^{1}\left(  \mathbb{B}^{2}\right)  $,%

\begin{equation}
\log\left(  \frac{1}{2\pi}\int_{\mathbb{S}^{1}}e^{u}\right)  \leq\frac{1}%
{4\pi}\int_{\mathbb{B}^{2}}\left\vert \nabla u\right\vert ^{2}+\frac{1}{2\pi
}\int_{\mathbb{S}^{1}}u. \label{OPS}%
\end{equation}
In \cite{W1} the following generalization was proved

\begin{theorem}
[\cite{W1}]Let $(\Sigma,g)$ be a smooth compact surface with nonnegative
Gaussian curvature and geodesic curvature $\kappa\geq1$. Then for any $u\in
H^{1}\left(  \Sigma\right)  $,%
\[
\log\left(  \frac{1}{L}\int_{\partial\Sigma}e^{f}\right)  \leq\frac{1}{2L}%
\int_{\Sigma}\left\vert \nabla f\right\vert ^{2}+\frac{1}{L}\int
_{\partial\Sigma}f.
\]
Here $L$ is the length of $\partial\Sigma$. Moreover if equality holds at a
nonconstant function, then $\Sigma$ is isometric to $\overline{\mathbb{B}^{2}%
}$ and all extremal functions are of the form%
\[
u\left(  x\right)  =\log\frac{1-\left\vert a\right\vert ^{2}}{1+\left\vert
a\right\vert ^{2}\left\vert x\right\vert ^{2}-2x\cdot a}+c,
\]
for some $a\in\mathbb{B}^{2}$ and $c\in\mathbb{R}$.
\end{theorem}

The argument in \cite{W1} is by a variational approach based on the inequality
(\ref{OPS}). We can deduce the above inequality directly from Corollary
\ref{inqf}. Indeed, taking $u=1+\frac{f}{q+1}$ in Corollary \ref{inqf} we
obtain%
\[
\left(  \frac{1}{L}\int_{\partial\Sigma}\left(  1+\frac{f}{q+1}\right)
^{q+1}\right)  ^{2/\left(  q+1\right)  }\leq\frac{\left(  q-1\right)
}{\left(  q+1\right)  ^{2}}\frac{1}{L}\int_{\Sigma}\left\vert \nabla
f\right\vert ^{2}+\frac{1}{L}\int_{\partial\Sigma}\left(  1+\frac{f}%
{q+1}\right)  ^{2}.
\]
This can be rewritten as
\begin{align*}
&  \left(  q+1\right)  \left\{  \exp\left[  \frac{2}{q+1}\log\frac{1}{L}%
\int_{M}\left(  1+\frac{f}{q+1}\right)  ^{q+1}\right]  -1\right\} \\
&  \leq\frac{\left(  q-1\right)  }{\left(  q+1\right)  }\frac{1}{L}%
\int_{\Sigma}\left\vert \nabla f\right\vert ^{2}+\frac{2}{L}\int
_{\partial\Sigma}f+\frac{1}{q+1}\frac{1}{L}\int_{\partial\Sigma}f^{2}.
\end{align*}
Letting $q\rightarrow\infty$ we get%
\[
\log\frac{1}{L}\int_{M}e^{f}\leq\frac{1}{2L}\int_{\Sigma}\left\vert \nabla
f\right\vert ^{2}+\frac{1}{L}\int_{\partial\Sigma}f.
\]

\end{document}